\documentclass[10pt, a4paper]{amsart}
\usepackage{url,graphicx}
\usepackage{multirow, xcolor}
\usepackage{fullpage}
\usepackage{amsfonts,amscd,amssymb,amsmath,amsthm,mathrsfs,multirow,lscape,xfrac}
\usepackage{pbox,lipsum, stmaryrd,url, epstopdf,float, enumerate}
\usepackage{array}
\usepackage{xspace, comment}
\usepackage[all,cmtip]{xy}
\usepackage[english]{babel}
\usepackage{color}

\usepackage{ucs}
\usepackage[utf8]{inputenc}
\usepackage{dsfont}
\usepackage[T1]{fontenc} 
\usepackage{lmodern}
\usepackage[thinlines]{easytable}

\usepackage{phaistos}
\usepackage{microtype}

\usepackage{enumitem}
\usepackage{tikz-cd}
\usepackage{scalerel}
\usepackage{todonotes}
\usepackage{stackengine,wasysym}
\usepackage{todonotes}
\usepackage[T1]{fontenc}
\usepackage[colorlinks,linkcolor=blue,anchorcolor=blue,citecolor=blue,backref=page]{hyperref}
\usepackage{cleveref}
\usepackage[english]{babel}

\newtheorem{theorem}{Theorem}[section]
\newtheorem{cor}[theorem]{Corollary}
\newtheorem{lemma}[theorem]{Lemma}

\theoremstyle{definition}

\newtheorem{remark}[theorem]{Remark}

\DeclareMathOperator{\Gal}{Gal}

\DeclareMathOperator{\Ban}{Ban}

\newcommand{\nc}{\newcommand}

\newcommand{\ner}{\mb{\tiny Ner}}

\newcommand{\Hom}{\mathrm{Hom}}
\newcommand{\End}{\mathrm{End}}
\newcommand{\Ext}{\mathrm{Ext}}

\newcommand{\beqar}{\begin{eqnarray*}}
\newcommand{\eeqar}{\end{eqnarray*}}

\newcommand{\oldmarginpar}[1]{}

\begin{document}
\nc{\Mhf}{{\bf M}_{0,H'}'}
\newcommand{\Muf}{{\bf M}'_{bal.U_1(p),H'}}
\newcommand{\Mho}{M_{0,H}}
\newcommand{\Muo}{M_{bal.U_1(p),H}}
\newcommand{\Mh}{M'_{0,H'}}
\newcommand{\Mu}{M'_{bal.U_1(p),H'}}
\newcommand{\Mucan}{{M'}_{bal.U_1(p),H'}^{{can}}}
\newcommand{\oMh}{\overline{M'}_{0,H'}}
\newcommand{\oMu}{\overline{M'}_{bal.U_1(p),H'}}
\newcommand{\pomega}{\omega^+}
\newcommand{\Eg}{{E'_1}\mid_S}
\newcommand{\Ig}{\mb{Ig}}

\newcommand{\Fab}{F_{\tiny \mb{abs}}}
\newcommand{\Frel}{F_{\tiny \mb{rel}}}
\nc{\sym}{\Psi}

\nc{\Gm}{\mathbb{G}_m}

\newcommand{\B}{\mathbb{B}}
\newcommand{\LT}{\mathcal{LT}}
\newcommand{\M}{\mathcal{M}}
\newcommand{\Ind}{\mathrm{Ind}}
\newcommand{\MH}{{\widehat{M}}'_{0,H'}}
\newcommand{\MU}{{\widehat{M}}'_{bal.U(p),H'}}
\newcommand{\Xfo}{{\widehat{X}}_!}
\newcommand{\Xfor}{{\widehat{X}}_{!!}}
\newcommand{\Xp}{\overline{X}_!}
\newcommand{\Xpp}{\overline{X}_{!!}}
\newcommand{\opi}{\overline{\pi}}
\newcommand{\ohat}[1]{\widehat{#1}}
\newcommand{\fsharp}{{\mathbf f}^\sharp_{{\tiny \Pi}}}
\newcommand{\mfrak}[1]{\mathfrak{#1}}
\newcommand{\bA}{{\bf A}}
\newcommand{\Mig}{\overline{M'}_{Ig,H'}}
\newcommand{\gp}{\mathfrak{p}}
\newcommand{\gP}{\mathfrak{P}}
\newcommand{\gl}{\mathfrak{l}}
\newcommand{\sA}{\mathcal{A}}
\newcommand{\et}{\acute{e}t}
\newcommand{\C}{\mathbb{C}}

\newcommand{\bF}{\mathbb{F}}
\newcommand{\cX}{\mathcal{X}}
\newcommand\A{\mathbb{A}}

\newcommand{\G}{\mathrm{G}}

\newcommand\T{\mathbb{T}}

\newcommand\sE{{\mathcal{E}}}
\newcommand\tE{{\mathbb{E}}}
\newcommand\sF{{\mathcal{F}}}
\newcommand\sG{{\mathcal{G}}}
\newcommand\GL{{\mathrm{GL}}}
\newcommand\bM{\mathbb{M}}
\newcommand\HH{{\mathrm H}}
\newcommand\J{\mathfrak{J}}
\newcommand\sT{\mathcal{T}}
\newcommand\Qbar{{\bar{\Q}}}
\newcommand\sQ{{\mathcal{Q}}}
\newcommand\sP{{\mathbb{P}}}
\newcommand{\Q}{\mathbb{Q}}
\newcommand{\cG}{\mathcal{G}}
\newcommand{\cM}{\mathcal{M}}
\newcommand{\tH}{\mathbb{H}}
\newcommand{\Z}{\mathbb{Z}}
\newcommand{\R}{\mathbb{R}}
\newcommand{\F}{\mathbb{F}}
\newcommand\fP{\mathfrak{P}}
\newcommand{\Mod}{\mathrm{Mod}}
\newcommand{\Ou}{\mathcal{O}}
\newcommand{\legendre}[2] {\left(\frac{#1}{#2}\right)}

\def \vp{\varphi}
\newcommand{\wh}[1]{\widehat{#1}}
\newcommand{\pip}{\Pi}
\newcommand{\Xdag}{\widehat{X}_\dagger}

\theoremstyle{definition}
\newtheorem{dfn}{Definition}

\theoremstyle{remark}

\newcommand {\stk} {\stackrel}
\newcommand{\map}{\rightarrow}
\def \bX{{\bf X}}
\newcommand {\orho}{\overline{\rho}}

\newcommand{\hGm}{\wh{\G}_m}

\newcommand{\Ner}[1]{{#1}^{\ner}}
\newcommand{\red}[1]{{\color{red} #1}}
\newcommand{\blue}[1]{{\color{blue} #1}}

\theoremstyle{remark}
\newtheorem*{fact}{Fact}

\makeatletter
\def\imod#1{\allowbreak\mkern10mu({\operator@font mod}\,\,#1)}
\makeatother

\title{A note on extensions of $p$-adic representations of $\mathrm{GL}_2(\mathbb{Q}_p)$} 
\author{Debargha Banerjee}
\author{Srijan Das}
\address{INDIAN INSTITUTE OF SCIENCE EDUCATION AND RESEARCH, PUNE, INDIA}
\thanks{The first author was partially supported by the SERB grant CRG/2020/000223. The second author was supported by the Prime Minister's Research Fellowship. It is a pleasure to acknowledge several helpful emails and advice from Professor Gabriel Dospinescu. We also wish to thank Professor Vytautas Pa{\v s}k\={u}nas for his comments on an early draft of the paper. The first author is indebted to MPIM for the $2023$ grant. The authors are grateful to the anonymous referee for a careful reading of the manuscript and for suggestions that have significantly improved the exposition.}

\begin{abstract}
We compute extension groups in the category of duals of $p$-adic Banach space representations of $\mathrm{GL}_2(\mathbb{Q}_p)$. Focusing on representations arising from the $p$-adic local Langlands correspondence for generic Galois representations, we classify these extensions completely. These results are then applied to prove the vanishing of extensions between the
dual $p$-adic Banach space representations attached to reducible Galois representations and supercuspidal Galois isotypic components of the $p$-adic \' etale cohomology of the finite level Drinfeld spaces. 
\end{abstract}

\maketitle
\begin{center} {\it This paper is dedicated to the fond memory of the first author's mother, the late Mrs. Snigdha Banerjee.} \end{center}

\section{Introduction} 
\label{Introduction}

Fix a prime $p$ and let $\G_{\Q_p}: = \Gal(\overline{\Q}_p/\Q_p)$ denote the absolute Galois group of $\Q_p$. The study of $\ell$-adic representations of $\G_{\Q_p}$, in the case where $\ell=p$, was pioneered by Fontaine, who laid the foundations of $p$-adic Hodge theory, and was further developed by Berger~\cite{BergerReduction} and Colmez~\cite{Colmez-Montreal}. The $p$-adic Langlands program, initiated by Breuil and further developed by Berger, Colmez, Dospinescu, Emerton, Kisin, and Pa{\v s}k\={u}nas, associates a $p$-adic admissible unitary Banach space representation $\Pi(V)$ of the group $\G := \GL_2(\Q_p)$ to a representation $V$ of $\G_{\Q_p}$ of dimension $\leq 2$.

While the bijection between the isomorphism classes of the above categories of representations is well-understood, a deeper understanding of the correspondence requires analyzing the derived category of representations; specifically, the extension groups between them. These extensions are central to the compatibility between the deformation theory of Galois representations and the deformation theory of $\G$-representations \cite{Paskunas13}.

Significant progress has been made in computing extension groups in the context of smooth mod $p$ representations of $\G$. Emerton \cite{Ordinary2} calculated the $\Ext^{\bullet}$ groups in the category of locally admissible mod $p$ representations of $\G$ as an application of his computation of the right derived functors of the ordinary parts. Breuil and Pa{\v s}k\={u}nas \cite{BreuilPaskunas} investigated similar extensions for $\GL_2(\Ou_F)$ (where $\Ou_F$ is the ring of integers of some finite extension $F/\Q_p$), while Colmez \cite{Colmez-Montreal} also performed foundational calculations for $\G$ in characteristic $p$. In \cite{Paskunasextensions}, Pa{\v s}k\={u}nas completed the computation of $\Ext^1(V_1, V_2)$ for absolutely irreducible mod $p$ representations $V_1, V_2$ (treating the previously open cases where $V_1$ is supersingular and $V_2$ is supersingular or a twist of Steinberg). 

However, in characteristic $0$, the literature is sparse. Hauseux \cite{MR3480966, MR3636753} has determined the extensions between unitary continuous $p$-adic and smooth mod $p$ principal series representations of $G'(F)$ in the generic case for split connected reductive groups $G'$. In \cite{Paskunas2adic}, Pa{\v s}k\={u}nas has calculated $\Ext^1$ between two supercuspidal representations of $\G$ in the $p$-adic setting.

The novelty of this short note lies in the computation of extensions in the category of duals of finite length admissible $p$-adic Banach space representations, specifically when one of the representations is dual of a non-ordinary representation. We say that a $p$-adic Banach space representation is non-ordinary if it is not a subquotient of a principal series representation. This dual category arises naturally in Colmez's construction of the $p$-adic local Langlands correspondence \cite{Colmez-Montreal}. If $D$ denotes the $(\varphi,\Gamma)$-module attached to a $p$-adic Galois representation $V$ via Fontaine's equivalence \cite{MR1106901}, then for $\delta := \varepsilon^{-1}\det V$ (where $\varepsilon$ is the $p$-adic cyclotomic character), the global sections of Colmez's sheaf $D \boxtimes_{\delta} \mathbb{P}^1$ fit into the fundamental exact sequence \cite{Colmez-Montreal}:
\begin{equation}\label{decompose_global_sections}
0 \to \Pi(V)^*\otimes\delta \to D\boxtimes_\delta\mathbb{P}^1 \to \Pi(V) \to 0.
\end{equation}
Here $\Pi(V)^*$ denotes the continuous dual equipped with the compact-open topology. The functor $\mathbf{V}$ (Colmez's Montr\'eal functor) and $\Pi$ satisfy the relation $\mathbf{V}(\Pi(V)) \cong V$.

Since the category of admissible representations with a central character on $p$-adic Banach spaces does not have enough injectives, we utilize Yoneda's construction to compute the extension groups in $\Ban_{\G,\zeta}^{\mathrm{adm}}(E)$ (cf. Sec \ref{padicLLC} for the definition), denoted by $\Ext_{\G, \zeta}^i$. Note that in the presence of enough injectives, Yoneda's definition is equivalent to the derived functor construction.

Our main result provides a complete classification of these extensions under a suitable genericity condition. Let \( E \) be a finite extension of \( \mathbb{Q}_p \), with ring of integers \( \mathcal{O} \), uniformizer \( \varpi \), and residue field \( \kappa := \mathcal{O}/(\varpi) \). We denote by \( \omega \) the mod $p$ cyclotomic character.

Let \( V \) be an absolutely irreducible two-dimensional \( E \)-representation of \( \mathrm{G}_{\mathbb{Q}_p} \), and let \( \overline{V} \) denote the semisimplification of its reduction modulo \( p \). We say that \( \overline{V} \) is \emph{generic} (in the sense of Pa\v{s}k\=unas~\cite{Paskunas13}) if \( \overline{V} \) is irreducible or \( \overline{V} \cong \psi_1 \oplus \psi_2 \) is split with characters \( \psi_1 \), \( \psi_2 \) satisfying $\psi_1 \psi_2^{-1} \notin \{1, \omega^{\pm 1}\}$.

\begin{theorem}
\label{main thm}
Let $V$ be an absolutely irreducible two-dimensional representation of $\G_{\Q_p}$ over $E$ such that its reduction $\overline{V}$ is generic. Let  $\pi$ be an absolutely irreducible admissible unitary $E$-Banach space representation of $\G$. If $\mathbf{V}(\pi) \neq V$, then
    \[
        \Ext_{\G,\zeta}^i(\pi^{*}, \Pi(V)^{*}) = 0 \quad \forall\, i \ge 0,
    \]
    where $\mathbf{V}$ denotes Colmez's Montr\'eal functor. 
Suppose that $\mathbf{V}(\pi)=V$. Then for every integer $i \ge 0$,
\[
\Ext_{\G,\zeta}^i\bigl(\pi^{*},\,\Pi(V)^{*}\bigr)\cong
\begin{cases}
E^{\binom{3}{i}}, & 0\le i\le 3,\\[4pt]
0, & i\ge 4.
\end{cases}
\]
\end{theorem}

This specific category of duals is also of significant arithmetic interest due to its geometric realization. Recall that Drinfeld spaces form a tower of rigid analytic covers of the $p$-adic upper half-plane $\Omega_{Dr,p} := \sP^{1}_{\Q_p} - \sP^1(\Q_p)$. This tower is defined over $\Breve{\Q}_p$, the completion of the maximal unramified extension of $\Q_p$, and is equipped with actions of $\GL_2(\Q_p)$, the units of the nonsplit quaternion algebra over $\Q_p$, and the Weil group $W_{\Q_p}$. Specifically, when $V$ is a two-dimensional supercuspidal representation of $\G_{\Q_p}$ (that is, the associated Weil-Deligne representation is irreducible), we derive from the work of Colmez, Dospinescu, and Nizio{\l} \cite{CDNJAMS, CDNForum} that the $V$-isotypic part of the $p$-adic \'etale cohomology of the finite level Drinfeld spaces (see Sec. \ref{CDNMain} for details) is isomorphic to a direct sum of copies of the dual $\Pi(V)^*$ (cf. isomorphism \ref{CDNforG}). Consequently, we determine that there is no extension of the $V$-isotypic part of the $p$-adic \'etale cohomology of these finite level Drinfeld spaces by the dual of the $p$-adic local Langlands associated with a reducible local Galois representation. Here, $\cM_{\infty}$ denotes the (non-complete) projective limit of the finite level Drinfeld spaces over $\C_p:=\widehat{\overline{\Q}}_p$ (defined precisely in Section \ref{CDNMain}).

\begin{cor}
\label{corollary ordinary}
Let $\rho:\Gal(\overline{\Q}/\Q) \rightarrow \GL_2(E)$ be a global Galois representation  with the corresponding local representation $\rho_p \cong \begin{pmatrix}
 \eta_1  & \star \\
 0 &  \eta_2\\
 \end{pmatrix} \otimes \eta$ where $\eta_1, \eta_2 :\Q_p^{\times} \rightarrow \mathcal{O}_E^{\times}$ are continuous integral characters and $\eta:\G_{\Q_p} \rightarrow E^{\times}$.  We also assume that $\eta_1 \cdot \eta_2^{-1} \not\in \{\varepsilon^{\pm 1}\}$. Suppose $\rho_p$ is potentially crystalline with distinct Hodge-Tate weights between $(0,k-1)$ for $k \geq 2$. Then, for any absolutely irreducible supercuspidal representation $V$ of $\G_{\Q_p}$ of dimension $\geq 2$  such that its mod $p$ reduction $\overline{V}$ is generic, we have 
\[\Ext_{\G,\zeta}^i(\Pi(\rho_p)^{*}, \HH_{\et}^1(\cM_{\infty}, E(1))[V])=0 \quad \text{for all } i \geq 0.\]
 \end{cor}

    \section{\texorpdfstring{$p$-adic and mod $p$  local Langlands for $\GL_2(\Q_p)$}{p-adic and mod p local Langlands for GL2(Qp)}} \label{padicLLC}

\subsection{The Montr\'eal Functor and the Trianguline Correspondence}
 Following \cite{MR3272011} and \cite{Paskunas13}, we recall some basic facts about the $p$-adic and mod $p$  local Langlands correspondence. As before, we fix a finite extension $E$ of $\Q_p$ with a ring of integers $\Ou$, a uniformizer $\varpi$, and a residue field $\kappa$.

Fix a continuous character $\zeta: Z \rightarrow \mathcal{O}^{\times}$ of the center $Z$ of $\G$. We denote by $\Mod^{\mathrm{sm}}_\G(\mathcal{O})$ the category of smooth $\mathcal{O}$-torsion $\G$-modules. Within this, let $\Mod^{\mathrm{sm}}_{\G,\zeta}(\mathcal{O})$ be the full subcategory consisting of representations on which $Z$ acts via $\zeta$. Furthermore, we define $\Mod^{\mathrm{fin}}_{\G, \zeta}(\mathcal{O})$ and $\Mod^{\mathrm{lfin}}_{\G, \zeta}(\mathcal{O})$ as the full subcategories of representations that are, respectively, of finite length and locally of finite length as $\mathcal{O}[\G]$-modules. A detailed study of these categories can be found in \cite[Section 2]{Ordinary1}.  In his Montr\'eal lectures, Colmez introduced an exact covariant $\Ou$-linear functor $\mathbf{V}: \Mod^{\mathrm{fin}}_{\G, \zeta}(\Ou)\rightarrow 
\Mod_{\G_{\Q_p}}(\Ou)$, where $\Mod_{\G_{\Q_p}}(\Ou)$ is the category of $\Ou$-modules with a continuous $\G_{\Q_p}$-action. Let $\Ban_{\G}^{\mathrm{adm}}(E)$ be the category of unitary admissible $E$-Banach space representations
of $\G$ and let $\Ban_{\G, \zeta}^{\mathrm{adm.fl}}(E)$ be the full subcategory consisting of finite length objects with central character $\zeta$. 
For any $\Pi \in \Ban_{\G}^{\mathrm{adm}}(E)$, one may choose an open bounded $\G$-invariant lattice $\Theta$ in $\Pi$. Then $\Theta/\varpi^n \Theta$ is a smooth admissible representation of $\G$ for all $n\ge 1$, 
implying it is locally finite. This allows one to define 
$\mathbf{V}(\Pi):=E\otimes \underset{\longleftarrow}{\lim}\, \mathbf{V}(\Theta/\varpi^n \Theta)$. Due to the commensurability of all open bounded lattices in $\Pi$, the definition is independent of the specific lattice $\Theta$ chosen. The functor $\mathbf{V}$ satisfies the following properties:
\begin{enumerate}
    \item It is exact and covariant from $\Ban_{\G, \zeta}^{\mathrm{adm.fl}}(E)$ to finite-dimensional continuous $E$-representations of $\G_{\Q_p}$.
    \item The functor $\Pi \mapsto \mathbf{V}(\Pi)$ induces a bijection between the isomorphism classes of absolutely irreducible non-ordinary $\Pi \in \Ban_{\G}^{\mathrm{adm}}(E) $ and two-dimensional absolutely irreducible continuous $E$-representations of $\G_{\Q_p}.$
    \item If $\psi: \Q_p^{\times}\rightarrow \Ou^{\times}$ is a continuous character, then we may also consider it as a continuous character 
$\psi: \G_{\Q_p}\rightarrow \Ou^{\times}$ via the local Artin map, and for all  $\Pi \in \Ban_{\G}^{\mathrm{adm}}(E)$, we have the isomorphism: $$\mathbf{V}(\Pi\otimes \psi\circ \det)\cong \mathbf{V}(\Pi)\otimes \psi.$$
\item The Montr\'eal functor maps ordinary Banach space representations of $\G$ to one-dimensional representations of $\G_{\Q_p}$. More precisely, as a functor from $\Mod^{\mathrm{fin}}_{\G, \zeta}(\Ou)$, $\mathbf{V}$ satisfies the following:
$$\mathbf{V}(\delta \circ \det)=0,\,\, \mathbf{V}(St \otimes \delta \circ \det)=\varepsilon \delta, \,\, \mathbf{V}(\Ind_{P}^{\G}{(\chi_1\otimes \chi_2\varepsilon^{-1})})\cong \chi_2;$$
where $P$ is the Borel subgroup of upper-triangular matrices in $\G$, $\Ind_{P}^{\G}$ denotes the (smooth) parabolic induction, $\chi_1$, $\chi_2$ and $\delta$ are continuous characters of $\Q_p^\times$, and $St$ is the Steinberg representation.
\end{enumerate}

Let $H \subset \G$ be a compact open subgroup, and let $\Ou[[H]]$ represent its completed group algebra. We define $\Mod^{\mathrm{pro \, aug}}_{\G}(\Ou)$ as the category of profinite, linearly topological $\Ou[[H]]$-modules $M$ equipped with a continuous $\G$-action, such that the action of $\Ou[H]$ induced by $\Ou[[H]]$ aligns with the restriction of the $\G$-action. Morphisms in this category are continuous $\G$-equivariant homomorphisms of topological $\Ou[[H]]$-modules. This definition is well-defined regardless of the choice of $H$, as any two compact open subgroups of $\G$ are commensurable. As noted in \cite[Equation 2.2.8]{Ordinary1}, taking Pontryagin dual yields an anti-equivalence between $\Mod^{\mathrm{sm}}_\G(\Ou)$ and $\Mod^{\mathrm{pro \, aug}}_\G(\Ou)$. The Pontryagin dual of an $\Ou$-module $M$ is defined as $$M^{\vee}:=\Hom^{cont}_{\Ou}(M, E/\Ou),$$ where $E/\Ou$ carries the discrete topology and $M^{\vee}$ the compact-open topology. There exists a canonical isomorphism $M^{\vee \vee}\cong M$.

Let $\mathfrak{C}(\Ou)$ be the full subcategory of $\Mod^{\mathrm{pro \, aug}}_\G(\Ou)$ whose objects consist of all $M$ isomorphic to $\pi^{\vee}$ for some $\pi \in \Mod^{\mathrm{lfin}}_{\G, \zeta}(\Ou)$. This category $\mathfrak{C}(\Ou)$ is abelian and has exact projective limits and projective envelopes. Let $\mathrm{Rep}_{\G_{\Q_p}}(\Ou)$ denote the category of continuous $\G_{\Q_p}$-representations on compact $\Ou$-modules. Pa{\v s}k\={u}nas \cite{Paskunas13} introduced a functor $\check{\mathbf{V}}: \mathfrak{C}(\Ou)\rightarrow \mathrm{Rep}_{\G_{\Q_p}}(\Ou)$ defined as follows: for an object $M \in \mathfrak{C}(\Ou)$ of finite length, set $\check{\mathbf{V}}(M):=\mathbf{V}(M^{\vee})^{\vee}(\varepsilon \zeta)$, where $\vee$ denotes the Pontryagin dual and $\zeta$ is treated as a character of $\G_{\Q_p}$ via local class field theory. More generally, by writing $M\cong \underset{\longleftarrow}{\lim}\, M_i$ as a limit over finite-length quotients in $\mathfrak{C}(\Ou)$, one defines $\check{\mathbf{V}}(M):=\underset{\longleftarrow}{\lim}\, \check{\mathbf{V}}(M_i)$. By construction, this functor is also covariant and exact. Under this normalization of $\check{\mathbf{V}}$, the following relations hold:
\begin{equation} \label{check-V}
     \check{\mathbf{V}}(\pi^{\vee})\cong \mathbf{V}(\pi), \quad \check{\mathbf{V}}((\Ind_{P}^{\G}{\chi_1\otimes \chi_2 \omega^{-1}})^{\vee})= \chi_1, \quad 
\check{\mathbf{V}}((St \otimes \eta\circ \det)^{\vee})=\eta,
\end{equation}
where $\pi$ represents a supersingular representation and $\eta$ is a continuous character of $\Q_p^\times$.

Consider $\Pi \in \Ban_{\G, \zeta}^\mathrm{adm}(E)$ with an open, bounded $\G$-invariant lattice $\Theta$. Its Schikhof dual is defined as $ \Theta^d:=\Hom_{\Ou}(\Theta, \Ou)$, endowed with the topology of pointwise convergence. It is the unit ball for the dual representation $\Pi^*$. The relationship between the Schikhof and Pontryagin duals is given by 
$$\Theta^d \otimes_{\Ou} \Ou/\varpi^n \Ou\cong (\Theta/\varpi^n \Theta)^{\vee}.
$$

According to \cite[Lemma 4.11]{Paskunas13}, $\Theta^d$ belongs to $\mathfrak{C}(\Ou)$. As $\Theta^d$ is $\Ou$-torsion free and $\check{\mathbf{V}}$ preserves exactness and covariance, $\check{\mathbf{V}}(\Theta^d)$ remains $\Ou$-torsion free. In this context, we define $\check{\mathbf{V}}(\Pi)$ as $\check{\mathbf{V}}(\Theta^d)\otimes_{\Ou} E$. The commensurability of different open lattices in $\Pi$ ensures that $\check{\mathbf{V}}(\Pi)$ is independent of the choice of $\Theta$.

\begin{lemma} \label{self-duality-of-LLC} 
Assume $\Pi$ and $\Theta$ are defined as above. Then $\check{\mathbf{V}}(\Pi)\cong \mathbf{V}(\Pi)^* (\varepsilon \zeta)$, where $*$ denotes the $E$-linear dual. Furthermore, if $\check{\mathbf{V}}(\Pi)$ is $2$-dimensional with determinant $\varepsilon \zeta$, then $\check{\mathbf{V}}(\Pi)\cong \mathbf{V}(\Pi)$.
\end{lemma}
\begin{proof}
    This is proved in \cite[Lemma 5.49]{Paskunas13}.
\end{proof}

The surjectivity of the functor $\mathbf{V}$ is established in \cite{Colmez-Montreal} for odd $p$ (and in \cite{MR3384443} for $p=2$) via the construction of a representation $\Pi(V)$ of $\G$ satisfying ${\bf V}(\Pi(V))=V$ for any $2$-dimensional representation $V$ of $\G_{\Q_p}$. The construction goes through Fontaine's equivalence \cite{MR1106901} relating representations of $\G_{\Q_p}$ and $(\varphi,\Gamma)$-modules, similar to the construction of $\Pi\mapsto {\bf V}(\Pi)$. The resulting $p$-adic local Langlands correspondence $V \mapsto \Pi(V)$ exhibits several notable features:
\begin{itemize}
  \item compatibility with reduction modulo $p$,
  \item consistent behavior within families,
  \item compatibility with the classical local Langlands correspondence.
\end{itemize}

An explicit construction of $\Pi(V)$ for trianguline $V$ is given in \cite[Section 6]{Emerton06}. We recall it here. Note that Emerton's description in the Banach space setting is the same as the characteristic zero translation of Colmez's mod $p$ results \cite[Th\'eor\`eme 0.11]{Colmez-Montreal} via the extension of the Montr\'eal functor to the Banach space setting described above. Throughout this construction, the superscript $\mathcal{C}^0$ indicates that the induction is taken in the category of continuous functions.

\begin{enumerate} \label{Emertonmain}
\item 
(Absolutely reducible)
Let $V \cong \begin{pmatrix}
 \eta_1 & 0  \\
 0 & \eta_2 \\
 \end{pmatrix} \otimes \eta$ with $ \eta_1,  \eta_2$ continuous {\it integral} characters and $\eta: \G_{\Q_p} \rightarrow E^{\times}$ be a continuous character. In this case, 
 \begin{itemize}
     \item If $\eta_1 \eta_2^{-1}\neq \varepsilon ^{\pm 1}$, then
     \begin{equation*}
         \Pi(V) \cong  \Ind^{\G}_{P}(\eta_1 \otimes \eta_2 \varepsilon^{-1})^{\mathcal{C}^0} \otimes \eta \bigoplus  \Ind^{\G}_ {P}(\eta_2 \otimes \eta_1 \varepsilon^{-1})^{\mathcal{C}^0} \otimes \eta. 
 \end{equation*} 
 \item If $\eta_1 \eta_2^{-1}= \varepsilon$ then
\[
 \Pi(V) \cong \eta_1 \circ \det \otimes B(2, \infty) \oplus \Ind^{\G}_ {P} (\eta_1 \varepsilon^{-1} \otimes \eta_1 \varepsilon)^{\mathcal{C}^0} \otimes \eta;
 \]
 where $B(2, \infty)$ denotes the universal unitary completion of the smooth induction
 \ $\Ind_{P}^{\G}(|\,|_{p}^{-1}\otimes |\,|_{p})^{sm}. $
 \end{itemize}

 \item 
 (Reducible non-split, case I)

 Let  $V \cong \begin{pmatrix}
 \eta_1 & \star \\
 0 & \eta_2 \\
 \end{pmatrix} \otimes \eta$  with $ \eta_1,  \eta_2 $ as above, $\star \neq 0$ and $\eta_1 \cdot \eta_2^{-1}\neq \varepsilon^{\pm 1}$, then $ \Pi(V)$ sits inside the following exact sequence in the category $\Ban_{\G}^{\mathrm{adm}}(E)$:
 \begin{equation*}
 \label{colmeznew}
 0 \rightarrow \pi_1 \otimes \eta  \rightarrow \Pi(V) \rightarrow \pi_2 \otimes \eta \rightarrow 0;
 \end{equation*}
 with  $\pi_1:=\Ind^{\G}_{ P}(\eta_2 \otimes \eta_1)^{\mathcal{C}^0}$ and $\pi_2:=\Ind^{\G}_{ P}(\eta_1 \varepsilon \otimes \eta_2 \varepsilon^{-1})^{\mathcal{C}^0}$.
 \item
(Reducible non-split, case II)
 Suppose  $V \cong \begin{pmatrix}
 \eta & \star \\
 0 & \eta \varepsilon ^{-1}\\
 \end{pmatrix} $ or $V \cong \begin{pmatrix}
 \eta \varepsilon ^{-1} & \star \\
 0 & \eta \\
 \end{pmatrix}$  with $\eta$ as above and  $\star \neq 0$. Then the corresponding $\G$-representation $\Pi(V)$ has a Jordan-H\"older filtration with Jordan-H\"older factors $\widehat{St}$, $\underline{1}$, and $\Ind_{P}^{\G} (\eta \varepsilon^{-1} \otimes \eta \varepsilon )^{\mathcal{C}^0}$ where $\widehat{St}$ denotes the universal unitary completion of the Steinberg representation $St$.

\end{enumerate}

\subsection{Blocks and semisimple mod $p$ LLC}

In \cite{Paskunas13}, Pa{\v s}k\={u}nas defined an equivalence relation on the set of (isomorphism classes of) irreducible objects of $\Mod^{\mathrm{lfin}}_{\G,\zeta}(\Ou)$, where $\tau \sim \pi$ if and 
only if there exists a sequence of irreducible representations $\tau=\tau_0, \tau_1, \ldots, \tau_n=\pi$ such that for each $i$, one of the following holds: $1)\,\tau_i=\tau_{i+1}$; 
$2)\,\Ext^1_\G(\tau_i, \tau_{i+1})\neq 0$; $3)\,\Ext^1_\G(\tau_{i+1}, \tau_i)\neq 0$. An equivalence class is called a block. The category $\Mod^{\mathrm{lfin}}_{\G, \zeta}(\Ou)$ is locally finite, and by general results of Gabriel \cite{MR232821}, it decomposes as a product 
$$\Mod^{\mathrm{lfin}}_{\G, \zeta}(\Ou)\cong \prod_{\mathfrak{B}} \Mod^{\mathrm{lfin}}_{\G, \zeta}(\Ou)^{\mathfrak{B}}$$ 
of indecomposable subcategories, where the product runs over the set of all blocks $\mathfrak{B}$ of $\Mod^{\mathrm{lfin}}_{\G,\zeta}(\Ou)$. Here, $\Mod^{\mathrm{lfin}}_{\G,\zeta}(\Ou)^{\mathfrak{B}}$ denotes the full subcategory 
of $\Mod^{\mathrm{lfin}}_{\G,\zeta}(\Ou)$ consisting of those representations such that all their irreducible subquotients lie in $\mathfrak{B}$. Dually, we obtain a decomposition:
\begin{equation*} 
\mathfrak{C}(\Ou)\cong \prod_{\mathfrak{B}} \mathfrak{C}(\Ou)^{\mathfrak{B}},
\end{equation*}
where $M\in \mathfrak{C}(\Ou)$ lies in $\mathfrak{C}(\Ou)^{\mathfrak{B}}$ if and only if $M^{\vee}$ lies in $\Mod^{\mathrm{lfin}}_{\G, \zeta}(\Ou)^{\mathfrak{B}}$. 

For each fixed block $\Mod^{\mathrm{lfin}}_{\G, \zeta}(\Ou)^{\mathfrak{B}}$, there is a finite extension $E'$ of $E$ with ring of integers $\Ou'$, such that $\Mod^{\mathrm{lfin}}_{\G, \zeta}(\Ou)^{\mathfrak{B}}\otimes_\Ou \Ou'$ decomposes into a finite product of indecomposable subcategories, each of which remains indecomposable after a further extension of scalars. Such absolutely indecomposable blocks have been classified in \cite{paskunasblocks}, which are described as follows:
\begin{enumerate}
   \item $\mathfrak{B}=\{\pi\}$, with $\pi$ supersingular,
   \item $\mathfrak{B}= \{\Ind_P^\G \chi_1 \otimes \chi_2 \omega^{-1}, \Ind_P^\G \chi_2 \otimes \chi_1 \omega^{-1} \}$, with $\chi_1 \chi_2 ^{-1} \notin \{1, \omega^{\pm 1}\}$,
   \item $\mathfrak{B}=\{ \Ind_P^\G \chi \otimes \chi \omega^{-1}\}$, with $p \geq 3$,
   \item $\mathfrak{B}=\{1, St\} \otimes \chi \circ \det$, with $p=2$,
   \item $\mathfrak{B}=\{1, St, \Ind_P^\G \omega \otimes \omega^{-1}\} \otimes \chi \circ \det$, with $p\geq 5$,
   \item $\mathfrak{B}=\{1, St, \omega \circ \det, St \otimes \omega \circ \det \} \otimes \chi \circ \det$, with $p=3$;
\end{enumerate}
where $\chi_1$, $\chi_2$ and $\chi$ are smooth characters of $\Q_p^\times$ and $\omega$ is the mod $p$ cyclotomic character given by $x \mapsto x|x|\pmod p$.

To each block above, one may attach a semi-simple $2$-dimensional $\kappa$-representation $\overline{\rho_{\mathfrak{B}}}$ of $\G_{\Q_p}$ using
the semi-simple mod $p$ correspondence of Breuil (\cite{Breuil03-1}, \cite{Breuil03-2}): in the case (1) 
$\overline{\rho_{\mathfrak{B}}}$ is absolutely irreducible, and  Colmez's functor $\mathbf{V}$ maps $\pi$ to $\overline{\rho_{\mathfrak{B}}}$, 
in the case (2) $\overline{\rho_{\mathfrak{B}}}=\chi_1\oplus \chi_2$, in cases (3) and (4) $\overline{\rho_{\mathfrak{B}}}= \chi\oplus \chi$, in cases (5) and (6) $\overline{\rho_{\mathfrak{B}}}=\chi\oplus 
\chi\omega$, where we consider characters of $\G_{\Q_p}$ as characters of $\Q_p^{\times}$ via local class field theory, normalized so that 
uniformizers correspond to geometric Frobenii. We note that the determinant of $\overline{\rho_{\mathfrak{B}}}$ is equal to $\zeta \varepsilon$ modulo $\varpi$, where $\varepsilon$ is the $p$-adic cyclotomic character.

\section{Vanishing Results for \texorpdfstring{$\Ext^{\bullet}$}{Ext^.} Groups}
\label{Jacquet}
We keep the notation from the previous sections: $E$ is a finite extension of $\mathbb{Q}_p$ with ring of integers $\Ou$, uniformizer $\varpi$, and residue field $\kappa$.
\subsection{Two dimensional pseudocharacters}\label{pseudocharacters}
We begin by reviewing the definition and some properties of $2$-dimensional pseudocharacters. Since these facts are standard, we quote much of this and the following paragraph directly from  \cite[Appendix A]{Paskunas13}. Let $\mathcal{G}$ denote a profinite group, and let $(A, \mathfrak{m}_A)$ be a local artinian $\Ou$-algebra. Throughout, we assume that $p>2$. A continuous map $T: \mathcal{G} \rightarrow A$ is defined as a $2$-dimensional $A$-valued pseudocharacter if it adheres to the following conditions:
\begin{enumerate}
    \item $T(1)=2$;
    \item $T(gh)=T(hg)$ for every $g,h \in \mathcal{G}$;
    \item The relation $T(g)T(h)T(k)-T(g) T(hk)- T(h) T(gk)- T(k) T(gh)+T(ghk)+T(gk h)=0$ holds for all $g,h, k \in \mathcal{G}$.
\end{enumerate}

It is a well-known result that for any continuous representation $\rho: \mathcal{G} \rightarrow \GL_2(A)$, the associated trace function $\mathrm{tr} \rho$ constitutes a $2$-dimensional pseudocharacter. Given a $2$-dimensional pseudocharacter $T: \mathcal{G}\rightarrow A$, one can define a continuous group homomorphism $\mathcal{G} \rightarrow A^{\times}$ via the formula $D(g):=\frac{T(g)^2-T(g^2)}{2}$ (see \cite[Proof of Prop. 7.23] {MR3444227}). As shown in \cite[Lemma 7.7, Prop. 7.23]{MR3444227}, the mapping $T\mapsto (T, D)$ establishes a bijection between the set of $2$-dimensional pseudocharacters and the set of pairs $(T, D)$, where $D:\mathcal{G} \rightarrow A^{\times}$ is a continuous group homomorphism and $T:\mathcal{G}\rightarrow A$ is a continuous function which satisfies: $T(1)=2$, $T(gh)=T(hg)$, and the identity $D(g)T(g^{-1}h)-T(g)T(h)+T(gh)=0$ for all $g, h\in \mathcal{G}$.

Consider a continuous representation $\rho:\mathcal{G}\rightarrow \GL_2(\kappa)$. Let $D_{\mathrm{tr} \rho}^{\mathrm{ps}}$ denote the functor from local artinian augmented $\Ou$-algebras with residue field $\kappa$ to the category of sets, where $D_{\mathrm{tr} \rho}^{\mathrm{ps}}(A)$ consists of all $2$-dimensional $A$-valued pseudocharacters $T$ lifting $\mathrm{tr}\rho$, that is satisfying the condition $T\equiv \mathrm{tr} \rho \pmod{\mathfrak{m}_A}$. Provided that $\Hom^{cont}(\mathcal{H}, \F_p)$ is a finite-dimensional $\F_p$-vector space for every open subgroup $\mathcal{H}$ of $\mathcal{G}$, this functor $D_{\mathrm{tr} \rho}^{\mathrm{ps}}$ is pro-represented by a complete local noetherian $\Ou$-algebra. Notably, this finiteness criterion is met when $\mathcal{G}$ is the absolute Galois group of a local field.

We usually work with a modified version where the determinant is fixed. Let $\psi: \mathcal{G}\rightarrow \Ou^{\times}$ be a continuous character that lifts $\det \rho$. We define $D_{\mathrm{tr} \rho}^{\mathrm{ps}, \psi}$ as the subfunctor of $D_{\mathrm{tr} \rho}^{\mathrm{ps}}$ where $T\in D_{\mathrm{tr} \rho}^{\mathrm{ps}, \psi}(A)$ if and only if $\frac{T(g)^2-T(g^2)}{2}$ coincides with (the image of) $\psi(g)$ for all $g\in \mathcal{G}$. We refer to $D_{\mathrm{tr} \rho}^{\mathrm{ps}, \psi}$ as a deformation problem with a fixed determinant. If $R_{\mathrm{tr} \rho}^{\mathrm{ps}}$ pro-represents $D_{\mathrm{tr} \rho}^{\mathrm{ps}}$, it follows that $D_{\mathrm{tr} \rho}^{\mathrm{ps}, \psi}$ is pro-represented by a quotient of this ring, which we denote by $R_{\mathrm{tr} \rho}^{\mathrm{ps}, \psi}$.

Recall that for the identity functor $I_\mathcal{C}$
of any category $\mathcal{C}$, the natural transformations $\eta: I_\mathcal{C} \rightarrow I_\mathcal{C}$
form a commutative monoid, called the center of the category $\mathcal{C}$. The following theorem characterizes the center of the category $\Mod^{\mathrm{lfin}}_{\G, \psi}(\Ou)^\mathfrak{B}$, as established in \cite[Theorem 1.3]{Paskunas2adic} (refer also to \cite{Paskunas13}).

\begin{theorem}\label{intro_centre}
Suppose the block $\mathfrak{B}$ is given by (1) or (2) above. There exists a natural isomorphism between the center of the category $\Mod^{\mathrm{lfin}}_{\G, \psi}(\Ou)^\mathfrak{B}$ and the ring $R_{\mathrm{tr}\overline{\rho_\mathfrak{B}}}^{\mathrm{ps}, \psi}$.
\end{theorem}

\subsection{Block decomposition in the Banach setting}

In the context of Banach space representation, for a block $\mathfrak{B}$, let $\Ban^{\mathrm{adm}}_{\G, \psi}(E)^\mathfrak{B}$ be the full subcategory of $\Ban^{\mathrm{adm}}_{\G, \psi}(E)$ consisting of those $\Pi$, such that for some (equivalently any) open bounded $\G$-invariant lattice $\Theta$ in $\Pi$, all the irreducible subquotients of $\Theta\otimes_{\Ou} \kappa$, which is an object of $\Mod^{\mathrm{lfin}}_{\G, \psi}(\Ou)$,  lie in $\mathfrak{B}$. It is shown in \cite[Prop. 5.36]{Paskunas13} that $\Ban^{\mathrm{adm}}_{\G, \psi}(E)$ decomposes into a direct sum of subcategories:
\begin{equation} \label{BlockDecomposition} 
\Ban^{\mathrm{adm}}_{\G, \psi}(E)\cong \bigoplus_{\mathfrak{B}} \Ban^{\mathrm{adm}}_{\G, \psi}(E)^\mathfrak{B},
\end{equation}
where the sum runs over the set of all blocks $\mathfrak{B}$ of $\Mod^{\mathrm{lfin}}_{\G,\psi}(\Ou)$. 

To understand the finer structure of these blocks, we first make explicit the action of the generic fiber of the relevant pseudodeformation ring. For any $\Pi \in \Ban^{\mathrm{adm}}_{\G, \psi}(E)^\mathfrak{B}$ equipped with an open bounded $\G$-invariant lattice $\Theta$, the quotients $\Theta/\varpi^n$ are objects of $\Mod^{\mathrm{lfin}}_{\G, \psi}(\Ou)^\mathfrak{B}$ for all $n \geq 1$. Theorem \ref{intro_centre} gives a natural action of $R_{\mathrm{tr}\overline{\rho_\mathfrak{B}}}^{\mathrm{ps}, \psi}$ on $\Theta/\varpi^n$. Upon taking the inverse limit and inverting $p$, we obtain a natural homomorphism $R_{\mathrm{tr}\overline{\rho_\mathfrak{B}}}^{\mathrm{ps}, \psi}[1/p]\rightarrow \End_\G^{\mathrm{cont}}(\Pi)$.

This action allows us to decompose the finite length objects. The Chinese remainder theorem, see \cite[Theorem 4.36]{Paskunas13}, yields an equivalence of categories:
\begin{equation} \label{CRT}
    \Ban^{\mathrm{adm}.\mathrm{fl}}_{\G, \psi}(E)^\mathfrak{B}\cong \bigoplus_{\mathfrak{n} \in \mathrm{maxSpec} R_{\mathrm{tr}\overline{\rho_\mathfrak{B}}}^{\mathrm{ps}, \psi}[1/p]} \Ban^{\mathrm{adm}.\mathrm{fl}}_{\G, \psi}(E)^\mathfrak{B}_{\mathfrak{n}},
\end{equation}
where, for a maximal ideal $\mathfrak{n}$ of $R_{\mathrm{tr}\overline{\rho_\mathfrak{B}}}^{\mathrm{ps}, \psi}[1/p]$, $\Ban^{\mathrm{adm}.\mathrm{fl}}_{\G, \psi}(E)^\mathfrak{B}_{\mathfrak{n}}$ consists of those finite length representations that are killed by a power of $\mathfrak{n}.$

\subsection{Proof of the main theorem}

Consider an absolutely irreducible representation $\Pi \in \Ban^{\mathrm{adm}}_{\G, \psi}(E)^\mathfrak{B}$. As established in the previous section, we have a natural homomorphism $x: R_{\mathrm{tr}\overline{\rho_\mathfrak{B}}}^{\mathrm{ps}, \psi}[1/p] \rightarrow \End_\G^{\mathrm{cont}}(\Pi)$. Because $\Pi$ is absolutely irreducible, Schur's Lemma \cite[Corollary 4.42]{Paskunas13} implies the isomorphism $\End_\G^{\mathrm{cont}}(\Pi) \cong E$. 

Finally, assuming the block $\mathfrak{B}$ corresponds to types (1) or (2) as defined in the preceding section, \cite[Cor. 2.25]{Paskunas2adic} provides the subsequent commutative diagram:

\begin{equation} \label{diag:lift-of-trace}
\begin{tikzcd}[column sep=large, row sep=large, ampersand replacement=\&]
    \& R_{\mathrm{tr}\overline{\rho_\mathfrak{B}}}^{\mathrm{ps}, \psi}[1/p] \arrow[d, "x"] \\
    \G_{\mathbb{Q}_p} \arrow[ru, "t^{\mathrm{univ}}"] \arrow[r, "\mathrm{tr}\check{\mathbf{V}}(\Pi)"'] \& E
\end{tikzcd}
\end{equation}
where $t^{\mathrm{univ}}$ is the universal pseudocharacter of $\G_{\Q_p}$. Let $\mathfrak{n}$ be the maximal ideal of $ R_{\mathrm{tr}\overline{\rho_\mathfrak{B}}}^{\mathrm{ps}, \psi}[1/p]$ corresponding to $\Pi$ in the above diagram. Then, by the construction of the specialization map $x$, we have $\mathfrak{n} \subseteq \mathrm{Ann}_{R_{\mathrm{tr}\overline{\rho_\mathfrak{B}}}^{\mathrm{ps}, \psi}[1/p]}(\Pi).$ Therefore, $\Pi$ is an element of $\Ban^{\mathrm{adm}.\mathrm{fl}}_{\G, \psi}(E)^\mathfrak{B}_{\mathfrak{n}}$.

\begin{proof} [{Proof of Theorem~\ref{main thm}}]
   Since $\pi$ and $\Pi(V)$ are absolutely irreducible, their continuous duals $\pi^*$ and $\Pi(V)^*$ are also absolutely irreducible. Therefore, by \cite[Theorem 1.4 and Corollary 1.6]{MR3272011}, they are of finite length. Now if $\pi^*$ and $\Pi(V)^*$ are in $\Ban_{\G,\zeta}^{\mathrm{adm}}(E)^{\mathfrak{B}}$ for different $\mathfrak{B}$, then by the block decomposition \ref{BlockDecomposition}, we are done. Assume they lie in $\Ban_{\G,\zeta}^{\mathrm{adm}}(E)^{\mathfrak{B}}$ for the same block $\mathfrak{B}$. By Lemma \ref{self-duality-of-LLC}, we have $\check{\mathbf{V}}(\Pi(V)^*)=V$, which is an absolutely irreducible representation of dimension $2$. Set $r := \check{\mathbf{V}}(\pi^*)$, which is also an absolutely irreducible $E$-representation of $\G_{\mathbb{Q}_p}$. Let $\mathfrak{n}$ be the maximal ideal of $R^{ps,\zeta\varepsilon}_{\mathrm{tr}\,\overline{\rho_{\mathfrak{B}}}}[1/p]$ corresponding to $\mathrm{tr}\,V$, and let $\mathfrak{m}$ be the maximal ideal of $R^{ps,\zeta\varepsilon}_{\mathrm{tr}\,\overline{\rho_{\mathfrak{B}}}}[1/p]$ corresponding to $\mathrm{tr}\,r$. Since $\overline{V}$ is generic, the above discussion implies $\Pi(V)^*\in \Ban_{\G,\zeta}^{\mathrm{adm}}(E)^{\mathfrak{B}}_{\mathfrak{n}}$ and $\pi^*\in \Ban_{\G,\zeta}^{\mathrm{adm}}(E)^{\mathfrak{B}}_{\mathfrak{m}}$. By the Brauer--Nesbitt theorem, the trace character is determined by the semisimplification, so $\mathrm{tr}\,V$ and $\mathrm{tr}\,r$ depend only on $V^{ss}\cong V$ and $r^{ss}\cong r$, respectively. Thus $\mathfrak{n}=\mathfrak{m}$ can occur only when $V \cong r.$

   If $\pi$ is ordinary, then by the isomorphisms in (\ref{check-V}), $r$ is one-dimensional, so $r$ can never be isomorphic to $V$. When $\pi$ is non-ordinary, the first isomorphism in (\ref{check-V}) tells us $r = \mathbf{V}(\pi)$. Therefore if $\mathbf{V}(\pi)\neq V$, then \cite[Corollary~6.8]{Paskunas-Tung} yields $\Ext_{\G,\zeta}^i(\pi^*,\Pi(V)^*)=0$ for all $i\ge0$.

   For the case $\mathbf{V}(\pi)=V=\mathbf{V}(\Pi(V))$, we have the equality $\mathfrak{n}=\mathfrak{m}$ between the maximal ideals of $R^{ps,\zeta\varepsilon}_{\mathrm{tr}\,\overline{\rho_{\mathfrak{B}}}}[1/p]$. Following the proof of \cite[Corollary~2.26]{Paskunas2adic}, we get that $\Ban^{\mathrm{adm}.\mathrm{fl}}_{\G, \zeta}(E)^\mathfrak{B}_{\mathfrak{n}}$ is anti-equivalent to the category of $A$-modules of finite length, where $A$ denotes the completion of $R^{ps,\zeta\varepsilon}_{\mathrm{tr}\,\overline{\rho_{\mathfrak{B}}}}[1/p]$ at $\mathfrak{n}$. Via this anti-equivalence, the objects \(\pi^*\) and \(\Pi(V)^*\) are both identified
with the residue field $\kappa(\mathfrak{n})$ of \(A\). Consequently, we obtain the identification:
   \begin{equation*}
       \Ext_{\G,\zeta}^i(\pi^*,\Pi(V)^*) \cong \Ext_A^i(\kappa(\mathfrak{n}),\kappa(\mathfrak{n})).
   \end{equation*}
 Moreover in the proof of \cite[Corollary~2.26]{Paskunas2adic}, it is shown that $A$ is a complete regular local ring of dimension $3$. Using the Koszul complex, we compute $\Ext_A^i(\kappa(\mathfrak{n}),\kappa(\mathfrak{n}))$, and by \cite[Exercise 4.5.6]{Weibel_1994} this group is a $\binom{3}{i}$-dimensional $\kappa(\mathfrak{n})$-vector space. As the residue field $\kappa(\mathfrak{n})$ is identified with $E$ via the specialization map $x$, we have the required dimension results.
   \end{proof}
\begin{remark}
   In the case that $V$ is residually irreducible with $\pi$ ordinary, the vanishing of $\Ext_{\G,\zeta}^1$ can be proved by purely automorphic methods using the results of \cite [Sec. 4]{Paskunas13} and \cite [Lemma 2.16] {Paskunas2adic}.  
\end{remark}

\section{Rational Cohomology of Drinfeld Space}
\label{CDNMain}

We recall the construction of the Drinfeld spaces as described in  \cite[\S 0.1]{CDNJAMS}.  For $\ell \neq p$, works of Faltings, Fargues, Harris, and Taylor establish that the \'etale cohomology of the Drinfeld spaces encodes both the classical local Langlands and the classical Jacquet-Langlands correspondence for $\GL_2(\Q_p)$. It is expected that the  $p$-adic \'etale cohomology groups similarly encode the hypothetical $p$-adic local Langlands.  

Let $\G: = \mathrm{GL}_2(\mathbb{Q}_p)$, $D$ denote the nonsplit quaternion algebra with center $\Q_p$, and $\check{\G}$ be the group of invertible elements of $D$. We write $\mathcal{O}_D$ for the maximal order of $D$, and $\varpi_D$ is a uniformizer of $\mathcal{O}_D$. The level structure is given by the sequence of subgroups  
\[
\check{\G}_n: =  
\begin{cases}  
\mathcal{O}_D^\times & \text{if } n = 0, \\
1 + \varpi^n_D \mathcal{O}_D & \text{if } n \geq 1.  
\end{cases}  
\]
Let $\Omega_{Dr,p}:=\sP^{1}_{\Q_p}-\sP^1(\Q_p)$ denote Drinfeld's $p$-adic upper half-plane. In \cite{MR0422290}, Drinfeld introduced certain covers $\Breve{\cM}_n$ of $\Omega_{Dr,p}$.
This covering is defined over $\Breve{\Q}_p:=\widehat{\Q^{nr}_p}$ and the action of the Weil group $W_{\Q_p}$ is compatible with the natural action of $\Breve{\Q}_p$.  There is a natural covering map $\Breve{\cM}_{n+1} \rightarrow \Breve{\cM}_n   \rightarrow \Omega_{Dr,p}$ compatible with the action of $\G$ and $\check{\G}$. The zeroth level of this tower is given by $ \Breve{\cM}_0= \mathbb{Z} \times \Omega_{\text{Dr},p}$, while for \(n \geq 1\), the space $\Breve{\cM}_n$ is a Galois cover of $\Breve{\cM}_0$, with the Galois group $\check{\G}_0/\check{\G}_n=\mathcal{O}_D^\times/(1 + \varpi_D^n \mathcal{O}_D)$. Define $\cM_{n,\C_p}:= \C_p \times_{\Breve{\Q}_p} \Breve{\cM}_n$ and let $\cM_{\infty}$ denote the (non-complete) projective limit of $\cM_{n, \C_p}$. We now study the structure of the $p$-adic \'etale cohomology of the Drinfeld spaces as a $\G$-representation.

 As before, let $E$ be a finite extension of $\Q_p$ with the residue field $\kappa$. Fix a de Rham representation $V$ of $\G_{\Q_p}:=\Gal(\overline{\Q}_p/\Q_p)$ of dimension $2$ with Hodge-Tate weights in $(0,1)$. One can associate to $V$ the filtered $E-(\varphi,N,\G_{\Q_p})$-module $ \mathbf{D}_{\mathrm{pst}}(V)$, a rank-two module over $E \otimes_{\Q_p} \Q_p^{\mathrm{nr}}$. By Fontaine's construction \cite{MR1293977}, we have the associated Weil–Deligne representation $\mathrm{WD}(V):=\mathrm{WD}(\mathbf{D}_{\mathrm{pst}}(V))$ of dimension $2$ over $E$. We write $\mathrm{LL}(V):=\mathrm{LL}(\mathrm{WD}(V))$ for the irreducible smooth $E$-representation of $\G$ arising via the classical local Langlands correspondence. It is an infinite-dimensional $\G$-representation and may be recovered as the space of $\G$-smooth vectors inside the $p$-adic local Langlands correspondence $\Pi(V)$. We say that $V$ is supercuspidal if $\mathrm{WD}(V)$ is irreducible (in which case 
$\mathrm{LL}(V)$ is supercuspidal). This implies, in particular, that $N = 0$ on 
$\mathbf{D}_{\mathrm{pst}}(V)$ (i.e.\ $V$ is potentially crystalline). Likewise, we denote by $\mathrm{JL}(V):=\mathrm{JL}(\mathrm{LL}(V))$ the smooth irreducible $E$-representation of $\check{\G}$, obtained via the Jacquet–Langlands correspondence. The representations $\mathrm{WD}(V)$, $\mathrm{LL}(V)$ and $\mathrm{JL}(V)$ depend only on the 
$E-(\varphi,N,\G_{\Q_p})$-module $\mathbf{D}_{\mathrm{pst}}(V)$, whereas $\Pi(V)$ genuinely depends on $V$ itself. Unlike $\mathrm{LL}(V)$, the representation $\mathrm{JL}(V)$ is finite‐dimensional over $E$.

The following fundamental result of Colmez–Dospinescu–Nizio{\l} (\cite [Thm. 0.2] {CDNJAMS}; see also \cite [Thm. 0.14]{CDNForum}) describes the $V$-isotypic component of the first $p$-adic \'etale cohomology of the finite-level Drinfeld spaces as a $\G \times \check{G}-$representation.

\begin{theorem}\label{cdn}
Let $V$ be an absolutely irreducible $E$-representation of $\G_{\Q_p}$ of dimension $\ge 2$.

\emph{(i)} 
If $V$ is supercuspidal of dimension~$2$ and of Hodge-Tate weights $0$ and $1$, then
\[
\Hom_{{\rm W}_{\Q_p}}\!\left(V,\, E\otimes_{\Q_p} \HH^1_{\et}({\mathcal {M}}_\infty,\Q_p(1))\right)
= {\mathrm {JL}}(V)\otimes_E \Pi(V)^*.
\]

\emph{(ii)}
In all other cases,
\[
\Hom_{{\rm W}_{\Q_p}}\!\left(V,\, E\otimes_{\Q_p} \HH^1_{\et}({\mathcal {M}}_\infty,\Q_p(1))\right)=0.
\]
\end{theorem}

\begin{proof} [{Proof of Corollary~\ref{corollary ordinary}}]
     As $\mathrm{JL}(V)$ is a finite-dimensional $E$-vector space, from the above Theorem \ref{cdn},we obtain the following isomorphism of $\G$-representations:
\begin{equation} \label{CDNforG}
   \HH^1_{\rm et}(\mathcal{M}_{\infty}, E(1))[V] \cong (\Pi(V)^*)^{\oplus r}
 \end{equation}
where $V$ is an absolutely irreducible $E$-representation of $\G_{\Q_p}$ and $$r=\begin{cases}
    \dim_E(\mathrm{JL}(V)) & \text{if $V$ is supercuspidal of dimension $2$ and H-T weights $0$ and $1$}, \\
    0 & \text{otherwise}.
\end{cases}$$
By the isomorphism \ref{CDNforG}, it is enough to prove that $\Ext_{\G,\zeta}^i(\Pi(\rho_p)^*, \Pi(V)^*)=0$ for all $i \geq 0$. From Emerton's description of $\Pi(\rho_p)$ for reducible $\rho_p$ in Section \ref{Emertonmain}, we get the exact sequence of duals:
    \begin{equation}
        0 \rightarrow  \pi_2^* \rightarrow \Pi(\rho_p)^* \rightarrow \pi_1^* \rightarrow 0, 
    \end{equation}
    where $\pi_1$ and $\pi_2$ are (twists of) principal series representations on a $p$-adic Banach space. Applying $\Hom_\G(\_,\Pi(V)^*)$ to the above short exact sequence, we get the following long exact sequence 
    \begin{align*}
        0 \rightarrow ...\rightarrow \Ext_{\G,\zeta}^i(\pi_1^*, \Pi(V)^*) \rightarrow \Ext_{\G,\zeta}^i(\Pi(\rho_p)^*, \Pi(V)^*) \rightarrow \Ext_{\G,\zeta}^i(\pi_2^*, \Pi(V)^*)\rightarrow ...
    \end{align*}
    Using Theorem~\ref{main thm} on this exact sequence, we get the result.
\end{proof}
\begin{remark}
   Recently, Vanhaecke \cite[Th\'eor\`eme 1.1]{vanhaecke2024cohomologie} extended Theorem \ref{cdn} to arbitrary weights by working with the $p$-adic étale cohomology with coefficients in the symmetric powers of the universal local system on the Drinfeld tower. In this broader setting, one readily verifies that our Corollary \ref{corollary ordinary} continues to hold.
\end{remark}

\bibliographystyle{crelle}
\bibliography{domp}
\end{document}